\documentclass[10pt,conference]{ieeeconf}

\usepackage{mathrsfs}
\usepackage{graphics} % for pdf, bitmapped graphics files
\usepackage{amsmath} % assumes amsmath package installed
\usepackage{amssymb}  % assumes amsmath package installed

\usepackage{amsthm}
\usepackage{cite}
\usepackage{bm}
\usepackage{acronym}
\usepackage{paralist}
\usepackage{float}
\usepackage{color}
\usepackage{epstopdf}
\usepackage{multicol}
\usepackage{tikz}
\usepackage{graphicx}
\usepackage{soul}
\usepackage{mathtools}
\usepackage{siunitx}
\usepackage{empheq}

\makeatletter
\let\NAT@parse\undefined
\makeatother

\usepackage{hyperref}

\makeatletter
\hypersetup{colorlinks=true}
\AtBeginDocument{\@ifpackageloaded{hyperref}
  {\def\@linkcolor{blue}
  \def\@anchorcolor{red}
  \def\@citecolor{red}
  \def\@filecolor{red}
  \def\@urlcolor{red}
  \def\@menucolor{red}
  \def\@pagecolor{red}
\begingroup
  \@makeother\`%
  \@makeother\=%
  \edef\x{%
    \edef\noexpand\x{%
      \endgroup
      \noexpand\toks@{%
        \catcode 96=\noexpand\the\catcode`\noexpand\`\relax
        \catcode 61=\noexpand\the\catcode`\noexpand\=\relax
      }%
    }%
    \noexpand\x
  }%
\x
\@makeother\`
\@makeother\=
}{}}
\makeatother

\newtheorem{Theorem}{Theorem}

\newtheorem{Lemma}{Lemma}

\newtheorem{Problem}{Problem}

\newtheorem{Remark}{Remark}

\newtheorem{Corollary}{Corollary}

\newtheorem{Assumption}{Assumption}

\newcommand{\bequ}{\begin{eqnarray}}
\newcommand{\eequ}{\end{eqnarray}}

\IEEEoverridecommandlockouts

\def\BibTeX{{\rm B\kern-.05em{\sc i\kern-.025em b}\kern-.08em
    T\kern-.1667em\lower.7ex\hbox{E}\kern-.125emX}}

\begin{document}

\title{\LARGE \bf A Quadratic Program based Control Synthesis under Spatiotemporal Constraints and Non-vanishing Disturbances}

\author{Mitchell Black \and Kunal Garg \and Dimitra Panagou
\thanks{
% This paragraph of the first footnote will contain the date on which you submitted your brief for review. It will also contain support  information, including sponsor and financial support acknowledgment. For  example, ``This work was supported in part by the U.S. Department of  Commerce under Grant BS123456.'' 
The authors would like to acknowledge the support of the Air Force Office of Scientific Research under award number FA9550-17-1-0284, National Science Foundation award number 1931982, and the Walter Byers Scholarship from the National Collegiate Athletic Association.}
\thanks{The authors are with the Department of Aerospace Engineering, University of Michigan, Ann Arbor, MI, USA; \texttt{\{mblackjr, kgarg, dpanagou\}@umich.edu}.}
}
\maketitle

\begin{abstract}
In this paper, we study the effect of non-vanishing disturbances on the stability of fixed-time stable (FxTS) systems. We present a new result on FxTS, which allows a positive term in the time derivative of the Lyapunov function with the aim to model bounded, non-vanishing disturbances in the system dynamics. We characterize the neighborhood to which the system trajectories converge, as well the time of convergence to this neighborhood, in terms of the positive and negative terms that appear in the time derivative of the Lyapunov function. Then, we use the new FxTS result and formulate a quadratic program (QP) that yields control inputs which drive the trajectories of a class of nonlinear, control-affine systems to a goal set in the presence of control input constraints and non-vanishing, bounded disturbances in the system dynamics. We consider an overtaking problem on a highway as a case study, and discuss how to setup the QP for the considered problem, and how to make a decision on when to start the overtake maneuver, in the presence of sensing errors. 
\end{abstract}

\section{Introduction}

Control design for systems with input and state constraints is not a trivial task. \textit{Spatio-temporal} specifications typically impose spatial constraints that require the system trajectories to be in a safe set at all times, and temporal constraints that impose convergence of the system trajectories to a goal set within a given time. Incorporating safety related constraints on the system states can be achieved via control barrier functions (CBF) \cite{ames2017control}. For requirements involving convergence of the system states to a desired location or a set, approaches using control Lyapunov functions (CLF) \cite{li2018formally,srinivasan2018control,ames2012control} are very popular. Many authors have used CLFs in control design either via Sontag's formula \cite{romdlony2016stabilization,garg2019prescribed}, or in an optimization framework \cite{ames2014rapidly,li2018formally} to guarantee convergence of closed-loop system trajectories to a given goal point or a goal set.

For concurrent safety and convergence guarantees, a combination of CLFs and CBFs in the control synthesis can be used \cite{ames2017control,romdlony2016stabilization}, where the CLF guarantees convergence while the CBF guarantees safety of the state trajectories. The authors in \cite{tee2009barrier} utilize Lyapunov-like barrier functions to guarantee asymptotic tracking of a time-varying output trajectory, while the system output always remains inside a given set. Casting control synthesis problems as quadratic programs has gained popularity recently due to ease of implementation on real-time systems \cite{cortez2019control,glotfelter2017nonsmooth}. The fact that CLF and CBF conditions are linear in the control input enables the use of QPs for problems involving spatiotemporal specifications \cite{li2018formally,srinivasan2018control,ames2017control}. The authors in \cite{lindemann2019control} use CBFs to encode signal-temporal logic (STL) specifications and formulate a QP to compute the control input. It is worth noticing that most of the aforementioned work is concerned with designing control laws so that reaching a desired location or a desired goal set is achieved as time goes to infinity, i.e., asymptotically.  

Based on the notion of fixed-time stability (FxTS) \cite{polyakov2012nonlinear}, the authors in \cite{garg2019prescribedTAC} define a Fixed-Time CLF to guarantee convergence of the state trajectories to the origin within a fixed time, as opposed to asymptotic or exponential convergence. From a practical point of view, it is also important to consider and design robust controllers against uncertainties and disturbances in the system dynamics to account for unmodeled dynamics and sensing errors. Robust CBFs guarantee forward-invariance of safe sets \cite{cortez2019control,nguyen2016orc,kolathaya2019isscbf}. Typically, the safe set is \textit{contracted} by a small amount that depends upon the Lipschitz constants of the CBF and the bound on the considered disturbance.

In the presence of non-vanishing disturbances, typically only boundedness of the trajectories in a neighborhood of the nominal equilibrium (or set) can be guaranteed (see, e.g., \cite[Section 9.2]{khalil2002nonlinear}). In this paper, we consider bounded, non-vanishing disturbances in the dynamics of a (nominal) system with a FxTS equilibrium, and guarantee that the system trajectories converge to a neighborhood of the nominal equilibrium point within a fixed time. We characterize the size of this neighborhood and the convergence time as a function of the bound of the considered disturbances. Then, in conjunction with robust CBFs, we formulate a QP to compute a control input that renders the safe set forward invariant, and drives the closed-loop trajectories to a neighborhood of a desired goal set within a fixed time, in the presence of control input constraints. Finally, to demonstrate the applicability of the theoretical results, we consider a two-lane overtake scenario where an \textit{Ego} car is required to overtake a \textit{Lead} car while maintaining a safe distance from it, within an available time-window dictated by the presence of an \textit{Oncoming} car in the overtake lane. We assume that the position and the velocity of the other cars are available to the Ego car within some bounded error to model sensing uncertainties, and that the control inputs are subject to some bounded actuation error. Then, utilizing the new robust FxTS result, we formulate a systematic way of deciding for the Ego car whether executing an overtake is safe or not. When safe, the developed QP formulation produces the controller for the Ego car to safely perform the overtake maneuver in the available time frame. 

The paper is organized as follows. Section II provides the foundations for Set Invariance and Fixed-Time Stability (FxTS), and introduces preliminary results on Robust FxT CLFs and Robust CBFs. In Section III an overtaking problem is used to motivate the Robust FxT-CLF-CBF-QP framework, while Section IV discusses the simulation results. We end with conclusion and directions for future work in Section V.

\section{Mathematical Preliminaries}\label{sec: math prelim}

In the rest of the paper, $\mathbb R$ denotes the set of real numbers and $\mathbb R_+$ denotes the set of non-negative real numbers. We use $\|\cdot\|$ to denote the Euclidean norm. We write $\partial S$ for the boundary of the closed set $S$, $\textrm{int}(S)$ for its interior. The Lie derivative of a function $V:\mathbb R^n\rightarrow \mathbb R$ along a vector field $f:\mathbb R^n\rightarrow\mathbb R^n$ at a point $x\in \mathbb R^n$ is denoted as $L_fV(x) \triangleq \frac{\partial V}{\partial x} f(x)$. 

\subsection{Forward Invariance of Safe Set}
Consider the control affine system
\begin{align}\label{cont aff sys}
% \begin{split}
    \dot x(t) & = f(x(t)) + g(x(t))u(t), \quad x(t_0) = x_0,
% \end{split}
\end{align}
where $x\in \mathbb R^n$, $u\in \mathcal U\subset \mathbb R^m$ are the state and the control input vector, respectively, $f:\mathbb R^n\rightarrow \mathbb R^n$ and $g:\mathbb R^n\rightarrow \mathbb R^{n\times m}$ are continuous functions. Here, $\mathcal U$ denotes the set of admissible control inputs. Define a safe set $S_s = \{x\; |\; h_s(x)\geq 0\}$, and consider a goal set to be reached in a prescribed time $T$ defined as $S_G = \{x\; |\; h_g(x)\leq 0\}$, where $h_s, h_g:\mathbb R^n\rightarrow\mathbb R$ are continuously differentiable functions.

We present a necessary and sufficient condition, known as Nagumo's Theorem, for guaranteeing forward invariance of the safe set $S_s$, i.e., safety of the system trajectories.

\begin{Lemma}\label{lemma: nec suff safety}
Let the solution of the \eqref{cont aff sys} exist and be unique in forward time. Then, the set $S_s$ is forward-invariant for the closed-loop trajectories of \eqref{cont aff sys} with $x(0)\in S_s$ if and only there exists a control input $u\in \mathcal U$ such that $L_fh_s(x) + L_gh_s(x)u\geq 0$ for all $x\in \partial S_s$, where $\partial S_{s} \triangleq \{x\; |\; h_s(x) = 0\}$ is the {boundary of the safe set $S_{s}$}. 
\end{Lemma}

\subsection{Fixed-Time Stability}
Next, we review the notion of fixed-time stability. Consider the nonlinear system
\begin{align}\label{ex sys}
\dot x(t) = f(x(t)), \quad x(0) = x_0,
\end{align}
where $x\in \mathbb R^n$ and $f: \mathbb R^n \rightarrow \mathbb R^n$ is continuous with $f(0)=0$. The origin is said to be an FxTS equilibrium of \eqref{ex sys} if it is Lyapunov stable and \textit{fixed-time convergent}, i.e., for all $x(0) \in \mathbb R^n$, the system trajectories satisfy $\lim_{t\to T} x(t)=0$, where $T <\infty$ is independent of $x(0)$ \cite{polyakov2012nonlinear}. Lyapunov conditions for FxTS is given as follows.

\begin{Theorem}[\hspace{-0.3pt}\cite{polyakov2012nonlinear}]\label{FxTS TH}
Suppose there exists a positive definite function $V:\mathbb R^n\rightarrow\mathbb R$ such that 
\begin{align}\label{eq: dot V FxTS old}
    \dot V(x) \leq -aV(x)^p-bV(x)^q,
\end{align}
holds along the trajectories of \eqref{ex sys} with $a,b>0$, $0<p<1$ and $q>1$. Then, the origin of \eqref{ex sys} is FxTS with a settling time $T \leq T_b$ where
\begin{align}\label{eq: T bound old}
    T_b \leq \frac{1}{a(1-p)} + \frac{1}{b(q-1)}. 
\end{align}
\end{Theorem}

\noindent We need the following lemma to prove one of the main results of the paper. 

\begin{Lemma}\label{lemma:int dot V non van}
Let $V_0, c_1, c_2>0, c_3>0$, $a_1 = 1+\frac{1}{\mu}$ and $a_2 = 1-\frac{1}{\mu}$, where $\mu>1$. Define 
\begin{align}\label{eq:int dot V van}
    I \triangleq \int_{V_0}^{\bar V}\frac{dV}{-c_1V^{a_1}-c_2V^{a_2}+c_3}.
\end{align}
Then, the following holds:
\begin{itemize}
    \item[(i)] If $c_3< 2\sqrt{c_1c_2}$, we have for all $V_0\geq \bar V = 1$
    \begin{align}\label{eq: I bound 1 van}
        I\leq  \frac{\mu}{c_1k_1}\left(\frac{\pi}{2}-\tan^{-1}k_2\right),
    \end{align}
    where $k_1 = \sqrt{\frac{4c_1c_2-c_3^2}{4c_1^2}}$ and $k_2 = \frac{2c_1-c_3}{\sqrt{4c_1c_2-c_3^2}}$;
    \item[(ii)] If $c_3 \geq 2\sqrt{c_1c_2}$ and $V_0\geq \bar V = k\Big(\frac{c_3+\sqrt{c_3^2-4c_1c_2}}{2c_1}\Big)^\mu$ with $k>1$, we have for all $V_0\geq \bar V$
    \begin{align}%\label{eq: I bound 3 van}
       \nonumber I & \leq \frac{\mu}{c_1(b-a)}\log\left(\frac{kb-b}{kb-a}\right),
    \end{align}
    where $a,b$ are the roots of $\gamma(z) \triangleq c_1z^2-c_3z+c_2 = 0$;
\end{itemize}
\end{Lemma}

\noindent The proof is provided in Appendix \ref{app proof lemma int dot nonvan}.

Building upon the \textit{nominal} system \eqref{cont aff sys}, we now consider the \textit{perturbed} system, given as 
\begin{align}\label{eq: nl aff disturbed}
    \dot x(t) = f(x(t)) + g(x(t))u(t) + \phi (x(t)), \quad x(0) = x_0
\end{align}

where $f,g$ are as in \eqref{cont aff sys}, and $\phi:\mathbb R^n\rightarrow\mathbb R^n$ is an added, unmatched disturbance, possibly non-vanishing, which is assumed to be bounded. We denote the upper bound as $\|\phi\|_\infty \triangleq \sup_{x\in \mathcal D_0}\|\phi(x)\|$, where $\mathcal D_0\subseteq\mathbb R^n$ is a neighborhood of the origin.
The added disturbance $\phi$ models uncertainties in the parameters used in the control design; external perturbations to the dynamics, such as wind; and actuation errors, for example a power surge. Although uncertainty in a system can be treated in several different ways (see, e.g., \cite{kolathaya2019isscbf}, \cite{xu2015robustness}), we will restrict our focus to systems of the form \eqref{eq: nl aff disturbed}. Next, we present a new result on robustness of the trajectories around nominal FxTS equilibria against a class of bounded, non-vanishing disturbances. 

\subsection{Robust FxT CLF}
We extend the result in Theorem \ref{FxTS TH} by introducing a positive constant in the upper bound of the time derivative of the Lyapunov candidate, $V$. We refer to $V$ as a robust FxT-CLF and . 

\begin{Theorem}\label{Th: FxTS new nonvan}
Let $V:\mathbb R^n\rightarrow \mathbb R$ be a continuously differentiable, positive definite, proper function, satisfying
\begin{align}\label{eq: dot V new ineq nonvan}
    \dot V \leq -c_1V^{a_1}-c_2V^{a_2}+c_3,
\end{align}
with $c_1, c_2>0$, $c_3\in \mathbb R$, $a_1 = 1+\frac{1}{\mu}$, $a_2 = 1-\frac{1}{\mu}$ for some $\mu>1$, along the trajectories of \eqref{ex sys}. Then, there exists a neighborhood $D$ of the origin such that for all $x(0)\in \mathbb R^n\setminus D $, the trajectories of \eqref{ex sys} reach $D$ in a fixed time $T$, where\small{
% \begin{align}
% \begin{align}
%     D & = \begin{cases}\{x\; |\; V(x)\leq \Big(\frac{c_3+\sqrt{c_3^2-4c_1c_2}}{2c_1}\Big)^\mu\}; & c_3\geq >2\sqrt{c_1c_2}, \\
%     \{x\; |\; V(x)\leq \frac{c_3}{2\sqrt{c_1c_2}}\}, & 0\leq c_3\leq <2\sqrt{c_1c_2},\\
%     \{0\}, & c_3\leq 0,
%     \end{cases},\label{eq: D neighbor expression}\\
%     T & \leq \begin{cases}\frac{\mu}{c_1(b-a)}\log\left(\frac{|1+b|}{|1+a|}\right); & c_3>2\sqrt{c_1c_2}, \\
%      \frac{\mu}{\sqrt{c_1c_2}}\Big(\frac{1}{k-1}\Big), & c_3 = 2\sqrt{c_1c_2}\\
%     \frac{\mu}{c_1k_1}\left(\frac{\pi}{2}-\tan^{-1}k_2\right), & 0\leq c_3<2\sqrt{c_1c_2},\\
%     \frac{\mu\pi}{2\sqrt{c_1c_2}}, & c_3\leq 0,
%     \end{cases},\label{new FxTS T est nonvan}
% \end{align}
\begin{align}
    D & = \begin{cases}\{x\; |\; V(x)\leq k\Big(\frac{c_3+\sqrt{c_3^2-4c_1c_2}}{2c_1}\Big)^\mu\}; & c_3\geq 2\sqrt{c_1c_2}, \\
    \{x\; |\; V(x)\leq \frac{c_3}{2\sqrt{c_1c_2}}\}, & 0 < c_3< 2\sqrt{c_1c_2},\\
    \{0\}, & c_3\leq 0,
    \end{cases},\label{eq: D neighbor expression}\\
    T & \leq \begin{cases}\frac{\mu}{c_1(b-a)}\log\left(\frac{kb-b}{kb-a}\right); & c_3\geq 2\sqrt{c_1c_2}, \\
    %  \frac{\mu}{\sqrt{c_1c_2}}\Big(\frac{1}{k-1}\Big), & c_3 = 2\sqrt{c_1c_2}\\
    \frac{\mu}{c_1k_1}\left(\frac{\pi}{2}-\tan^{-1}k_2\right), & 0< c_3<2\sqrt{c_1c_2},\\
    \frac{\mu\pi}{2\sqrt{c_1c_2}}, & c_3\leq 0,
    \end{cases},\label{new FxTS T est nonvan}
\end{align}}\normalsize
where $k>1$, $a,b$ are the solutions of $\gamma(s) = c_1s^2-c_3s+c_2 = 0$,  $k_1 = \sqrt{\frac{4c_1c_2-c_3^2}{4c_1^2}}$, and $k_2 = -\frac{c_3}{\sqrt{4c_1c_2-c_3^2}}$.% and $\bar V = \frac{c_3}{2\sqrt{c_1c_2}}$. 
\end{Theorem}
\begin{proof}
Note that for $c_3\leq 0$, we obtain \eqref{eq: dot V FxTS old} from \eqref{eq: dot V new ineq nonvan}, and so FxTS of the origin is guaranteed for all $x\in \mathbb R^n$. Thus, we concentrate on the case when $c_3>0$, for which sufficiently small values of $V$ cause the right hand side of \eqref{eq: dot V new ineq nonvan} to become positive. The proof follows from Lemma \ref{lemma:int dot V non van}. Consider \eqref{eq: dot V new ineq nonvan} and let $c_1V^{a_1}+c_2V^{a_2} > c_3$. Re-write the inequality to obtain
\begin{align}\label{eq: dot V int form}
    \int_{V_0}^{V(x(T))}\frac{1}{-c_1V^{a_1}-c_2V^{a_2}+c_3}dV \geq \int_{0}^Tdt = T,
\end{align}
where $V_0 = V(x(0))$ and $T$ is the time when the system trajectories first reach the domain $D$. It is easy to show that for each of the cases listed in the Theorem statement, $c_1V^{a_1}+c_2V^{a_2} > c_3$ and thus the right-hand side of \eqref{eq: dot V new ineq nonvan} is negative for all $x\notin D$. Now, to show that the system trajectories converge to $D$ in fixed-time, we compute upper bounds on $T$.

For the case when $c_3<2\sqrt{c_1c_2}$, part (i) in Lemma \ref{lemma:int dot V non van} provides an upper bound on the left-hand side of \eqref{eq: dot V int form} for $x\notin D = \{x\; |\; V(x)\leq \frac{c_3}{2\sqrt{c_1c_2}}\}$. Similarly, for the case $c_3 \geq 2\sqrt{c_1c_2}$, part (ii) of Lemma \ref{lemma:int dot V non van} provides upper bounds on the left-hand side of \eqref{eq: dot V int form}. Thus, we obtain the domains $D$ and the bounds on convergence times $T$ for the various cases directly from Lemma \ref{lemma:int dot V non van}. Since for all three cases, $T<\infty$ and is independent of the initial conditions, we have that the system trajectories reach the set $D$ within a fixed time $T$.
\end{proof}

Next, we use Theorem \ref{Th: FxTS new nonvan} to show robustness of a FxTS origin against a class of non-vanishing, bounded, additive disturbance in the system dynamics.   

\begin{Corollary}\label{cor: robust phi FxTS nonvan}
Assume that there exists $u(t) \in \mathcal U$, where $\mathcal U$ is a set of admissible control inputs, such that the origin for the \textit{nominal} system \eqref{cont aff sys} is fixed-time stable, and that there exists a Lyapunov function $V$ satisfying conditions of Theorem \ref{FxTS TH}. Additionally, assume that there exists $L>0$ such that $ \left\|\frac{\partial V}{\partial x}\right\|\leq L$ for all $x\in \mathcal D_0\subseteq \mathbb R^n$. Then, there exists $D\subset \mathbb R^n$ such that for all $x(0)\in \mathcal D_0\setminus D$, the trajectories of \eqref{eq: nl aff disturbed} reach the set $D$ in a fixed time. 
\end{Corollary}
\begin{proof}
The time derivative of $V$ along the system trajectories of \eqref{eq: nl aff disturbed} reads
\begin{align*}
    \dot V = \frac{\partial V}{\partial x} [ f(x) + g(x)u + \phi(x)]\leq &-aV^p-bV^q + L\|\phi\|_\infty.
    % \\    \leq & -aV^p-bV^q + \frac{k_4L}{k_3}.
\end{align*}
Hence, using Theorem \ref{Th: FxTS new nonvan}, we obtain that there exists $D\subset \mathbb R^n$ such that all solutions starting outside $D$ reach the set $D$ in a fixed time $T$, where the set $D$ and the convergence time $T$ is a function of $a, b, p, q, L$ and $\|\phi\|_\infty$.
\end{proof}

Note that in the presence of non-vanishing disturbances, it is not possible to guarantee that the system trajectories converge to the equilibrium point. Instead, \eqref{eq: D neighbor expression} characterizes an estimate, $D$, of a neighborhood of the equilibrium to where system trajectories are guaranteed to converge within a fixed-time, $T$, and \eqref{new FxTS T est nonvan} provides an upper bound independent of $x(0) \in D$ on $T$. We observe that although this result shares commonalities with the notion of Input-to-State Stability \cite{liberzon1999iss}, it is both more restrictive on $\dot{V}$ and allows us to explicitly characterize $D$ and $T$.

\subsection{Robust CBF}
Next, we review the notion of a robust CBF, to guarantee forward invariance of a safe set, in the presence of a class of additive, non-vanishing disturbances. Here, we assume that $S_s\subset\mathcal D_0$. 

\begin{Lemma}
The set $S_s$ is forward-invariant for the closed-loop trajectories of \eqref{eq: nl aff disturbed} if 
\begin{align}\label{eq: h dot nonvan safety}
    \inf_{u\in \mathcal U}\{L_fh_s(x) + L_gh_s(x)u\}\geq - \left\|\frac{\partial h_s(x)}{\partial x}\right\|\|\phi\|_\infty,
\end{align}
holds for all $x \in \partial S_s\cap \mathcal D_0$. 
\end{Lemma}
\begin{proof}
The time derivative of $h_s$ along the trajectories of \eqref{eq: nl aff disturbed} reads
\begin{align*}
    \dot h_s = L_fh_s(x)+L_gh_s(x)u + \frac{\partial h_s(x)}{\partial x}\phi(x).
\end{align*}
For $x\in \partial S_s\cap \mathcal D_0$, we have that $h_s(x) = 0$ and $\|\phi(x)\|\leq \|\phi\|_\infty$. Using \eqref{eq: h dot nonvan safety}, we obtain that there exists a $u\in \mathcal U$ such that $\dot h_s\geq 0$.
% \begin{align*}
%     \dot h_s \leq \frac{\partial h_s(x)}{\partial x}\phi(x)- \left\|\frac{\partial h_s(x)}{\partial x}\right\|\|\phi\|_\infty \leq 0.
% \end{align*}
Thus, using Lemma \ref{lemma: nec suff safety}, we have that forward invariance of set $S$ is guaranteed. 
\end{proof}

Thus, condition \eqref{eq: h dot nonvan safety}, which notably need only hold at the boundary of a safe set, $S_s$, can be used to guarantee forward invariance of such a set in the presence of a class of additive, non-vanishing disturbances. Next, we take up a case study, and discuss how we can use the robust FxT-CLF and robust CBF in a QP framework to compute a control input so that the conditions \eqref{eq: dot V new ineq nonvan}, \eqref{eq: h dot nonvan safety} hold along the closed-loop trajectories.

\section{Case study: Overtake Problem}

In this section, we introduce a framework for computing overtake control via a FxT-CLF-CBF QP subject to bounded, non-vanishing, additive disturbances.

\subsection{Problem Formulation}
We consider an Ego car starting behind a slowly-moving, Lead car on a two-lane undivided highway, where the Ego car seeks to overtake the Lead car in a safe, timely manner whilst avoiding oncoming traffic (see Figure \ref{fig:problem setup}). The combined effort to achieve lane-keeping (maintaining the vehicle's position at the center of the lane), obstacle avoidance (remaining a safe distance between both other vehicles and the road edges), and goal-reaching within a fixed-time, $T$, (completing the overtake) in the presence of input constraints makes this problem challenging.

% The challenges specific to the problem include lane keeping (maintaining the vehicle's position at the center of the lane), obstacle avoidance (maintaining safe distance from the Lead vehicle and the Oncoming vehicles), and goal-reaching (overtake) within a fixed time, $T$, all while adhering to input constraints.

\begin{figure}[!ht]
    \centering
        \includegraphics[width=0.9\columnwidth,clip]{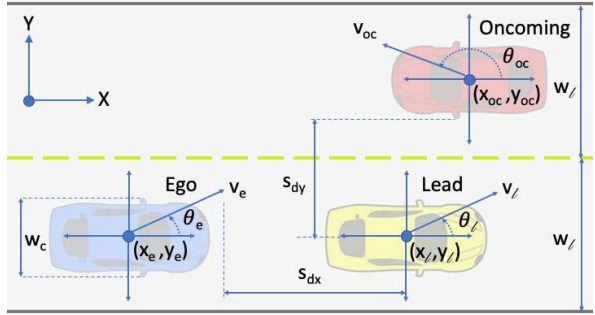}
    \caption{Problem setup for the overtake problem. The Ego car seeks to overtake the Lead car safely in the overtaking lane while avoiding a collision with the Oncoming car.}\label{fig:problem setup}
\end{figure}

For each vehicle, we select the model of a kinematic bicycle in an inertial frame, introduced in \cite{Rajamani2012VDC} and adapted for automobile highway merging in \cite{huang2018highwaymerging}. We use subscripts $e, l, oc$ to denote the Ego, the Lead and Oncoming car. The motion of the cars is modelled as:\small{
\begin{align}\label{eq: overtake_dynamics}
    \dot q_i = 
        \begin{bmatrix}
            v_i \ \cos(\theta_i) \\
            v_i \ \sin(\theta_i) \\
            0 \\  0
        \end{bmatrix}
        + 
        \begin{bmatrix}
            0 & 0 \\
            0 & 0 \\
            1 & 0 \\
            0 & 1
        \end{bmatrix}
        \begin{bmatrix}
            \omega_i \\ a_i\end{bmatrix}+ \phi_i,
\end{align}}\normalsize
where $q_i = [x_i \ y_i \ \theta_i \ v_i]^T$ is the state vector of car $i\in \{e, l, oc\}$, $x_i$ is the longitudinal position, $y_i$ is the transverse position, $\theta_i$ is the heading angle, $v_i$ is the velocity, $\omega_i$ is the angular control input, and $a_i$ is the longitudinal control input (measured as a fraction of $M_ig$, where $M_i$ is the vehicle mass and $g = 9.81$ m/s$^2$). The disturbance in each car's dynamics, $\phi _i$, takes into account modelling error and external perturbations such as wind or road grade. We assume that the disturbance $\phi _i$ is bounded, i.e., if $\hat q_l^e, \hat q_{oc}^e$ denote the states of the Lead and Oncoming car as estimated by Ego car, then there exists $\epsilon>0$ such that $\|\hat q_j^e(t)-q_j(t)\|\leq \epsilon$ for all $t\geq 0$, $j \in \{l, oc\}$. Consistent with the discussion in the previous section, we define $\|\phi\|_\infty = \epsilon$.

The control input $u_i\in \mathbb R^2$ for car $i$ consists of $\omega_i$ and $a_i$. Notably, our adjustment to the dynamics of \cite{Rajamani2012VDC} is such that $\theta$ describes the full steering dynamics, $\theta = \frac{v\tan(\beta)}{l_v}$, where $\beta$ is the steering angle in rad and $l_v$ is the length of the vehicle in m. This is a reasonable modification due to the small angle approximation, which we expect to hold in our overtaking problem, and from which we obtain that $tan(\beta) \approx \beta$, such that $\theta \approx \frac{v\beta}{l_v}$. Additionally, the vehicles are assumed to obey the no-slip condition imposed by the kinematic bicycle model, and their volumes are taken into consideration when evaluating safety. 

The overtake problem considered in the case study is formally stated below.

\begin{Problem}\label{Prob: Automobile Overtake}
Given $q_e(0), q_l(0), q_{oc}(0)$ determine if overtaking the Lead car is safe, i.e., if there exist vehicle state and control trajectories, $q_e(t), q_l(t), q_{oc}(t), u_e(t)$, where $u_e(t)\in \mathcal U = \{(\omega, a)\; |\; \omega_m\leq \omega\leq\omega_M, a_m\leq a\leq a_M\}$, such that $\|x_e(t) - x_i(t)\| > s_{dx}, \|y_e(t) - y_i(t)\| > s_{dy}$ for $i \in \{l,oc\}$ and $t \in [0,T]$, where $T$ is the upper bound on time required to complete the overtake. If safe, design a control input, $u_e(t)\in \mathcal U$ for all
$x_e(0) < x_{l}(0) - s_{dx}$, $y_e(0) = y_{l}(0)$, $x_{l}(0) < x_{oc}(0)$, $\theta _e(0) = \theta _{l}(0) = 0$, $\theta _{oc}(0) = \pi$, and $v_e(0), v_{l}(0), v_{oc}(0) > 0$, so that the closed-loop trajectories of the Ego car overtake the Lead car.
\end{Problem}

We divide the Problem \ref{Prob: Automobile Overtake} into the following sub-problems:
% {\color{red}
% \begin{enumerate}
% \item \textit{Determine Safety of Overtake Maneuver}: Using estimates of the time required to perform the full overtake maneuver and estimates of the positions of the Lead and Oncoming vehicles, in the presence of bounded measurement errors, determine the time at which an overtake maneuver is safe to initiate;
% \item \textit{Ego Vehicle Merge into Overtaking Lane}: Steer the Ego Vehicle into the overtaking lane while maintaining a safe distance with the Lead car;
% \item \textit{Ego Vehicle Move a Safe Distance beyond Lead Vehicle}: Advance the Ego vehicle a safe distance beyond the Lead vehicle in the overtaking lane. 
% \item \textit{Ego Vehicle Merge back into Original Lane}: Steer the Ego vehicle back into the original lane of travel while maintaining a safe distance from the Lead vehicle.
% \end{enumerate}
% }

% \renewcommand{\labelenumi}{\alph{enumi}}
\begin{enumerate}
\item Determine when an overtake is safe to initiate;
\item  Steer Ego Vehicle safely into overtaking lane;
\item  Advance Ego Vehicle safely past Lead Vehicle;
\item  Steer Ego Vehicle safely back into original lane.
\end{enumerate}

Note that lane maintenance can also be modelled as a safety constraint. The following CBFs were designed as such: $h_{s,1}$ encodes that the Ego vehicle maintains all four wheels within the road limits at all times even with bounded steering capabilities, while $h_{s,2}$ encodes that the Ego vehicle maintain a safe distance from the Lead Vehicle, as defined by the ellipse centered on the Lead Vehicle with semi-major axes $s_{dx}$ and $s_{dy}$ for the $x$ and $y$ coordinates respectively. The CBFs $h_{s,i}(q)$ were defined as follows:
\begin{align}
    h_{s,1} & = (y_e - e_1)(y_e - e_2)\\
    h_{s,2} & = 1 - (\frac{x_{l} - x_e}{s_{dx}})^2 - (\frac{y_{l} - y_e}{s_{dy}})^2&&
\end{align}
where $s_{dx} = v_e \tau \cos{\theta _e} + l_c$, $s_{dy} = w_l - \frac{w_c}{2}$, and $e_1$, $e_2$ are parameters which define the safety barrier at the edge of the road in the $y$ coordinate. Here, $\tau = 1.8$ sec is the time headway\footnote{$\tau = 1.8$ sec comes from the "half-the-speedometer" rule, as in \cite{xu2015robustness}.}.  
Specifically, we define $e_1 = (\frac{w_c}{2}) + v_e\omega _{max}\left(1 - \cos{\theta _e}\right)$, and $e_2 = (2w_l - \frac{w_c}{2}) - v_e\omega _{max}\left(1 - \cos{\theta _e}\right)$, where $w_l = 3$m is the width of a lane\footnote{Taken from  \url{https://tinyurl.com/knzhwje}} and $w_c = 2.27$m and $l_c = 5.05$m are the width and length of a car\footnote{Taken from \url{ https://tinyurl.com/y2rr375y}}. 

To capture the convergence requirement in each of the sub-problems 2) -- 4), we define goal sets $S_{G_j} = \{q\; |\; V(q-q_j^g)\leq 0\}$, where $q_j^g = [x_j^g\; y_j^g\; \theta_j^g\;v_j^g]^T$ denotes the goal location for the $j-th$ sub-problem, $j\in \{2, 3, 4\}$, and we define $\bar q_j = q-q_j^g$. We use a CLF $V:\mathbb R^4\rightarrow\mathbb R$ to encode the convergence requirement, defined as
\begin{align}\label{ex: CLF}
\begin{split}
    V(\bar q) = & K (k_x\bar x^2 + k_v\bar v^2 + k_{xv}\bar x\bar v \\ & + k_y\bar y^2 + k_\theta\bar \theta^2 + k_{y\theta}\bar y\bar \theta -1)\\
\end{split}
\end{align}
where $K$ is a constant gain selected during our parameter selection phase and $k_x$, $k_y$, $k_\theta$, $k_v$ $> 0$ are constant gains which influence the size and shape of the goal subspace. For positive definiteness of $V$, we impose that $0 < k_{xv} < 2\sqrt{k_xk_v}$ and $0 < k_{y\theta} < 2\sqrt{k_yk_{\theta}}$. Finally, we denote $\bar x =x_e-x_j^g,\bar y = y_e-y_j^g,\bar \theta = \theta_e-\theta_j^g$ and $\bar v =v_e-v_j^g$. Note that the convergence requirement, and thus the CLF $V$, changes for each sub-problem.

Consider the following inequalities.
\begin{align}
    L_fh_{s_1}(x)+L_gh_{s_1}(x)u& \geq 0, \label{eq: assum safety S1}\\
    L_fh_{s_2}(x)+L_gh_{s_2}(x)u& \geq 0, \label{eq: assum safety S2}
\end{align}
which when true at the boundary of the sets $S_{s_1}$ and $S_{s_2}$ guarantee forward invariance of the respective sets. We need the following viability assumption before we can proceed with our main results. 

\begin{Assumption}\label{Assum feas}
There exists a control input $u\in \mathcal U$ such that 
\begin{itemize}
    \item[1)] for all $q\in \partial S_{s_1}\cap \partial S_{S_2}$, both \eqref{eq: assum safety S1} and \eqref{eq: assum safety S2} holds;
    \item[2)] for all $q\in \partial S_{s_1}$ (respectively, $q\in \partial S_{s_2}$, \eqref{eq: assum safety S1} (respectively, \eqref{eq: assum safety S2}) holds.
\end{itemize}
Furthermore, $S_{G_j}\cap S_{s_1}\cap S_{s_2}\neq \emptyset$,  $\forall j \in \{2,3,4\}$. 
\end{Assumption}

We will now introduce a QP formulation to solve Problem \ref{Prob: Automobile Overtake}. Consider the QP:
\begin{subequations}\label{Robust FxT CLF-CBF QP}
\begin{align}
    \min_{u,\delta _1,\delta _2,\delta _3} \frac{1}{2}u^{T}u +p_1\delta _1^2 &+ p_2\delta _2^2 +p_3\delta _3^2 +q_1\delta _1 \label{pt_J}\\
    \textrm{s.t.} \; A_{u}u & \leq b_u \label{pt_c}\\
    L_fV(q_e)+L_gV(q_e)u & \leq  \delta _1 - \alpha_1 \max\{0,V(q_e)\}^{\gamma _1} \nonumber \\ 
    &\quad - \alpha_2\max\{0,V(q_e)\}^{\gamma _2} \label{pt_V}\\
    L_fh_{s_1}(q_e)+L_gh_{s_1}(q_e)u & \geq -\delta _2h_{s_1}(q_e) \label{pt_S1}\\
    L_fh_{s_2}(q_e)+L_gh_{s_2}(q_e)u & \geq -\delta _3h_{s_2}(q_e) \label{pt_S2}
\end{align}
\end{subequations}
where ($\ref{pt_J}$), quadratic in the decision variables, models a minimum-norm controller with relaxation variables $\delta _1$, $\delta _2$, $\delta _3$ and $p_1$, $p_2$, $p_3$, $q_1$ $\geq 0$, $\gamma _1 = 1 + \frac{1}{\mu}$, $\gamma _2 = 1 - \frac{1}{\mu}$, where $\mu > 1$, and $\alpha _i = \frac{\pi\mu}{2T}$ for $i = \{1,2\}$. The constraints, all of which are linear in the decision variables, accomplish the following: ($\ref{pt_c}$) enforces input constraints, ($\ref{pt_V}$) provides the FxT convergence guarantee, and ($\ref{pt_S1}$) and ($\ref{pt_S2}$) provide safety guarantees. Our formulation, specifically ($\ref{pt_V}$), utilizes the result of Theorem \ref{Th: FxTS new nonvan} in order to guarantee fixed-time convergence for any $\delta _1$. Moreover, we discuss the relationship between this $\delta _1$ term and an upper limit on the class of additive, bounded, non-vanishing disturbances considered in Problem \ref{Prob: Automobile Overtake}. 

Next, we discuss the feasibility of the QP \eqref{Robust FxT CLF-CBF QP}.

\begin{Lemma}
Under Assumption \ref{Assum feas}, the QP \eqref{Robust FxT CLF-CBF QP} is feasible for all $q\in (S_{S_1}\cap S_{S_2})\setminus S_G$.
\end{Lemma}
\begin{proof}
Let $q\notin S_G$, and consider the three cases $q\in \textnormal{int}(S_{S_1})\cap \textnormal{int}(S_{S_2})$, $q\in \partial S_{S_1}$ and $q\in \partial S_{S_2}$, separately.

In the first case, we have that $h_{s_1}, h_{s_2}, V\neq 0$. Choose any $u$ that satisfies \eqref{pt_c}. With this choice of $u$, one can choose $\delta_1, \delta_2, \delta_3$ so that \eqref{pt_V}-\eqref{pt_S2} hold with equality. This is possible since functions $V, h_{s_1}, h_{s_2}$ are non-zero. Thus, for all $q\in \left(\textnormal{int}(S_{S_1})\cap \textnormal{int}(S_{S_2})\right)\setminus S_G$, there exists a solution to \eqref{Robust FxT CLF-CBF QP}. Per Assumption \ref{Assum feas}, for all $q\in \partial S_{S_1}$, there exists a control input $u\in \mathcal U$ such that \eqref{pt_S1} holds with any $\delta_2$ (since $h_{s_1}(q) = 0$ for $q\in \partial S_{S_1}$, the choice of $\delta_2$ does not matter). Thus, using any $u$ that satisfies \eqref{pt_S1}, one can define $\delta_1$ and $\delta_3$ so that \eqref{pt_V} and \eqref{pt_S2} hold with equality. Similarly, one can construct a solution for the case when $q\in \partial S_{S_2}$, and $q\in \partial S_{S_1}\cap \partial S_{S_2}$. Thus, the QP \eqref{Robust FxT CLF-CBF QP} is feasible for all $q\in (S_{S_1}\cap S_{S_2})\setminus S_G$. 
\end{proof}

We are now ready to present our main result.

\begin{Theorem}\label{Th: Robust QP}
Let the solution to the QP \eqref{Robust FxT CLF-CBF QP} be denoted as $z^*(\cdot) = (u^*(\cdot), \delta _1^*(\cdot), \delta _2^*(\cdot), \delta _3^*(\cdot))$. Assume that $\|\phi\|_{\infty} \leq \frac{\delta _1^*}{\|\frac{\partial V}{\partial q}\|}$ for all $q_e$, i.e. $L_\phi V \leq \delta _1^*(q)$.
If the  solution $z^*(\cdot)$ is continuous on $\left (S_{S_1}\cap S_{S_2}\right ) \setminus S_G$, then under the effect of the control input $u(q_e) = u^*(q_e)$, the closed-loop trajectories of \eqref{eq: nl aff disturbed} reach a neighborhood $D$ of the goal set $S_g$ in fixed-time $T$, and  satisfy $q_e(t)\in S_{S_1}\cap S_{S_2}$ for all $t\geq 0$, where the neighborhood $D$ and time of convergence are given by \eqref{eq: D neighbor expression} and \eqref{new FxTS T est nonvan}, with $c_1= \alpha_1, c_2 = \alpha_2$ and $c_3 = 2\max \delta _1^*$.
\end{Theorem}

\begin{proof}
The proof for the unperturbed case is immediate. The constraint \eqref{pt_V} ensures that the conditions of Theorem \ref{Th: FxTS new nonvan} are satisfied and therefore convergence to the neighborhood $D$ is achieved in fixed-time, $T$, for the \textit{nominal} system $\dot q = f(q) + g(q)u$. For the \textit{perturbed} system, $\dot q = f(q) + g(q)u + \phi (q)$, we have that $\dot V = L_fV + L_gVu + L_{\phi}V \leq -\alpha _1V^{\gamma _1} - \alpha _2V^{\gamma _2} + \delta_1^*$, which may be rewritten as $\dot V = L_fV + L_gVu \leq -\alpha _1V^{\gamma _1} - \alpha _2V^{\gamma _2} + 2\max \delta_1^*$.
% Then $\dot q = f(q) + g(q)u + \phi (q)$ is fixed-time stabilizable if for all $t \geq 0$, $\dot V = \inf_{u \in \mathbb U} \{L_fV + L_gVu + L_{\phi}V\} \leq -\alpha _1V^{\gamma _1} - \alpha _2V^{\gamma _2} + \delta_1^*$, or similarly, using the upper bound on $\phi (q)$ as a worst case consideration, if $\dot V = \inf_{u \in \mathbb U} \{L_fV + L_gVu\} \leq -\alpha _1V^{\gamma _1} - \alpha _2V^{\gamma _2} + 2\max_{q_e}\{\delta_1^*(q_e)\}$. 
Thus, we have that the closed-loop trajectories of $\dot q = f(q) + g(q)u + \phi (q)$ reach $D$ in fixed-time $T$, given by \eqref{eq: D neighbor expression} and \eqref{new FxTS T est nonvan}, respectively, with $c_1 = \alpha _1$, $c_2 = \alpha _2$, $c_3 = 2\max \delta_1^*$.
\end{proof}

\begin{Remark}
Comparing \eqref{pt_V} and Theorem \ref{Th: FxTS new nonvan} yields an observation that $\delta _1^*$ in the solution of \eqref{Robust FxT CLF-CBF QP} is analogous to $c_3$ in \eqref{eq: dot V new ineq nonvan}. However, in the context of solving Problem \ref{Prob: Automobile Overtake}, \eqref{Robust FxT CLF-CBF QP} must be point-wise in the state space. It follows, therefore, that by considering $\max \delta _1^*$ over the solution set of \eqref{Robust FxT CLF-CBF QP} we can use $c_3 = 2\max \delta _1^*$ to obtain a useful, albeit conservative, estimate for the settling time to a neighborhood, $D$, of the goal set, $S_g$. 
\end{Remark}

\begin{Remark}
This method does not estimate the disturbance term, $\phi$; rather, it determines a tolerable upper bound such that FxTS to a neighborhood $D$ of a goal set $S_G$ is preserved, as well as characterizations of $D$ and the convergence time, $T$.
\end{Remark}

Next, we introduce a method for conditioning the parameters in \eqref{Robust FxT CLF-CBF QP} such that $2\max_{q} \delta _1(q)^* \leq 2\sqrt{\alpha _1 \alpha _2}$. 
We then use $c_3 = 2\max _q \{\delta _1^*\}$ to compute a conservative estimate on settling time for the Ego Vehicle during each segment of Problem \ref{Prob: Automobile Overtake}. 
We use the sum total of these time estimates to compute an unsafe overtaking horizon, i.e. $\left (v_e\cos \theta _e - \hat v_{oc}\cos{\hat \theta _{oc}}\right ) T_{est}$. If an Oncoming Vehicle is inside of the overtaking horizon (nearer to the Ego Vehicle than the horizon), then the Ego Vehicle does not begin its overtake.

\section{Simulation Results and Discussion}\label{sec: results}

\subsection{Simulation Parameters}

The CLF gains in \eqref{ex: CLF} are fixed as: $k_x = \frac{1}{60^2}$m$^{-2}$, $k_y = 100$m$^{-2}$, $k_\theta = 400$rad$^{-2}$, $k_v = 1$(m/s)$^{-2}$, $k_{xv} = 0.05\sqrt{k_xk_v} = \frac{1}{1200}$m$^{-2}$s, $k_{y\theta} = 0.5\sqrt{k_yk_{\theta}} = 100$(rad m)$^{-1}$ so that the goal set is defined as $C_g$: $\| \bar x \| \leq 60$ m, $\| \bar y \| \leq 0.1$ m, $\| \bar \theta \| \leq 0.05$ rad, $\| \bar v \| \leq 1$ m/s. The physical boundaries of the road are set to be $y = 0$ and $y = 2w_l$ respectively. The input constraints are given as $|\omega| \leq \frac{\pi}{18}$ rad and $|a| \leq 0.25g$ ms$^{-2}$. We used a time-step of $dt = 0.001$ sec. Other simulation parameters are: $\mu = 5$, which leads to $\gamma _1 = 1.2$ and $\gamma _2 = 0.8$, as well $p_1 = 1200$, $p_2 = 1$, and $q_1 = 1000$. We define $\theta _g = \tan^{-1} (\frac{y_g - y_e}{x_g - x_e})$ and set $v_g = 25$ as soon as it is safe to overtake. The final states of one segment are used as initial states to the subsequent segment. 

The following discussion will outline in greater detail the setup for each sub-problem. 

\subsubsection{Ego Vehicle Identify Opportunity to Perform Overtake}
\noindent The initial states of the Ego ($q_e(0)$), Lead ($q_{l}(0)$), and Oncoming ($q_{oc}(0)$) are chosen as $q_{e}(0) = 
        \begin{bmatrix}
            -\tau v_{l}(0) &
            \frac{w_l}{2} &
            0 &
            v_{l}(0)
        \end{bmatrix}^T
    $, $q_{l}(0) = 
        \begin{bmatrix}
            \tau v_{l}(0) &
            \frac{w_l}{2} &
            0 &
            v_{l}(0)
        \end{bmatrix}^T
    $, and $q_{oc}(0) = 
        \begin{bmatrix}
            x_{e}(0) + 2v_{l}(0)t_p &
            2w_l - \frac{w_l}{2} &
            -\pi &
            25
        \end{bmatrix}$
where $v_{l}(0)$ and $t_p$, the time until the Oncoming Vehicle passes by the Ego Vehicle, are chosen a priori. The goal state, $q_g$, is defined as an evolving function of $q_{l}$ where the goal location is chosen as $x_g = x_{l} - 1.5\tau v_{l} + 50$, $y_g = y_{l}$, and $v_g = v_l$ until an overtake maneuver is safe to initiate.

\subsubsection{Ego Vehicle Merge into Overtaking Lane}
\noindent We define $y_g$ and $v_g$ in this segment as: $y_g = y_{l} + w_l$, $v_g = 25$. The upper bound on settling time, $T$, is set to $T=10$ sec.

\subsubsection{Ego Vehicle Move a Safe Distance beyond Lead Vehicle}
\noindent The $x_g$ coordinate is modified to be: $x_g = x_{l} + 1.5\tau v_{l} + 50$, and $T=\frac{2\tau v_{l}(0)}{v_g - v_{l}(0)} + 4$ sec to adjust for an increase in safe following distance at increased initial velocities.

\subsubsection{Ego Vehicle Merge back into Original Lane}
\noindent We set $y_g = y_{l}$, and $T = 6$sec.

\subsection{Results}
In accordance with Theorem \ref{Th: FxTS new nonvan}, we desire to choose parameters such that it is guaranteed that $2\max \delta _1 = c_3 < 2 \sqrt{c_1 c_2}$, where $c_i = \alpha _i = \frac{\pi \mu}{2T}$ for $i \in \{1,2\}$ for our \textit{nominal} simulation. The initial conditions chosen are $v_{l}(0) = 17$ and $t_p = 2$. Thus, we varied $K$ from $10^{-5}$ to 1, $T$ from 13.15 to 30.65, $\omega_{max}$ from 0.0175 to 14.45, and $a_{max}$ from 0.245 to 245.25. We selected the final values as $K = 0.0001$, $T = 27.65$, and $u_{max} = [0.1745 \; 2.45]^T$. From Figure \ref{fig:1} we see that while increasing control authority yields a marginal decrease in $c_3$, there is a more considerable decrease in $c_3$ as the fixed-time window increases. Continuing to increase time for the sake of reducing $c_3$, however, reaches a point where it is no longer practical. As such, we selected $2\max \delta _1^* = 0.638$. To model the perturbation, we chose a zero-mean, Gaussian normal distribution with $3\sigma = \|\phi\|_{\infty}$ and saturated at $\pm \|\phi\|_{\infty}$, where $\sigma$ is the standard deviation with $\|\phi\|_{\infty} = 3.99$. 

\begin{figure}[!ht]
    \centering
        \includegraphics[width=1\columnwidth,clip]{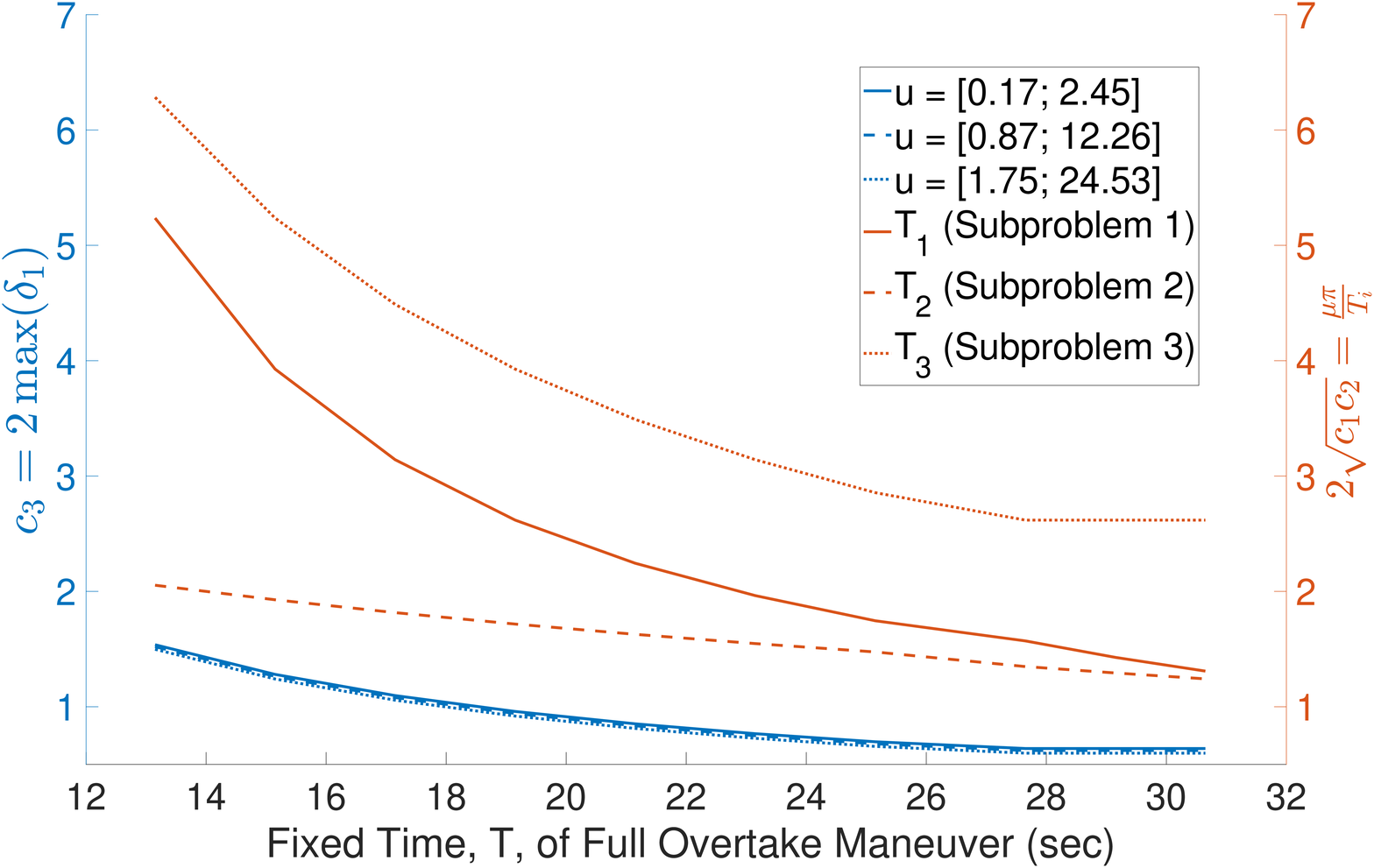}
    % \caption{Effect of QP parameters on $2\max\delta_1$ and $2\sqrt{\alpha _1\alpha _2}$, specifically for various fixed-times, $T$, and control input bounds $u_{max}$.}\label{fig:1}
    \caption{The effect of control authority, $u_{max}$, and fixed-time bound, $T$, on parameter selection for a FxT-CLF-CBF-QP controller in the presence of bounded model perturbations.}\label{fig:1}
\end{figure}

\begin{figure}[!ht]
    \centering
        \includegraphics[width=1\columnwidth,clip]{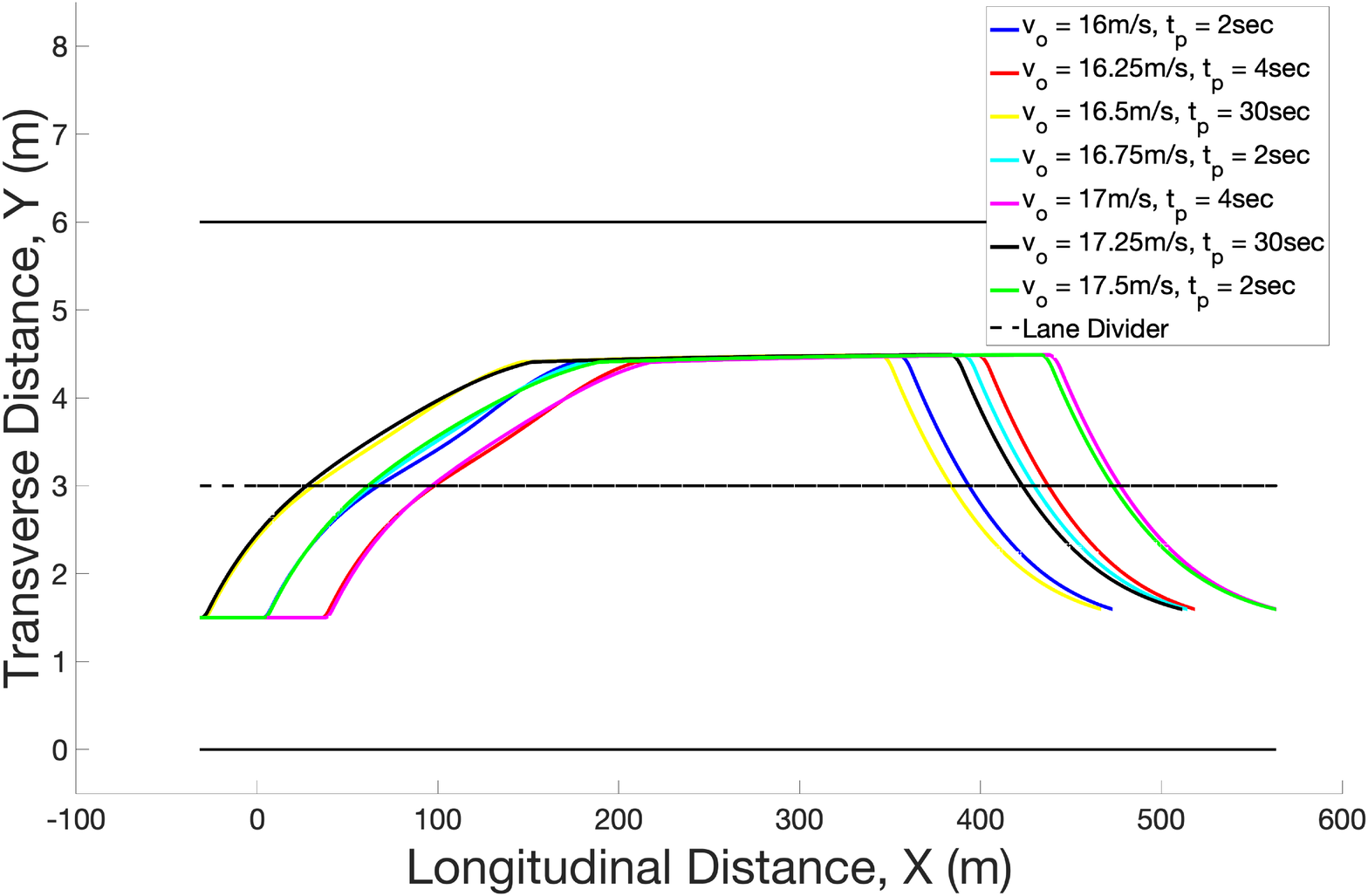}
        \includegraphics[width=1\columnwidth,clip]{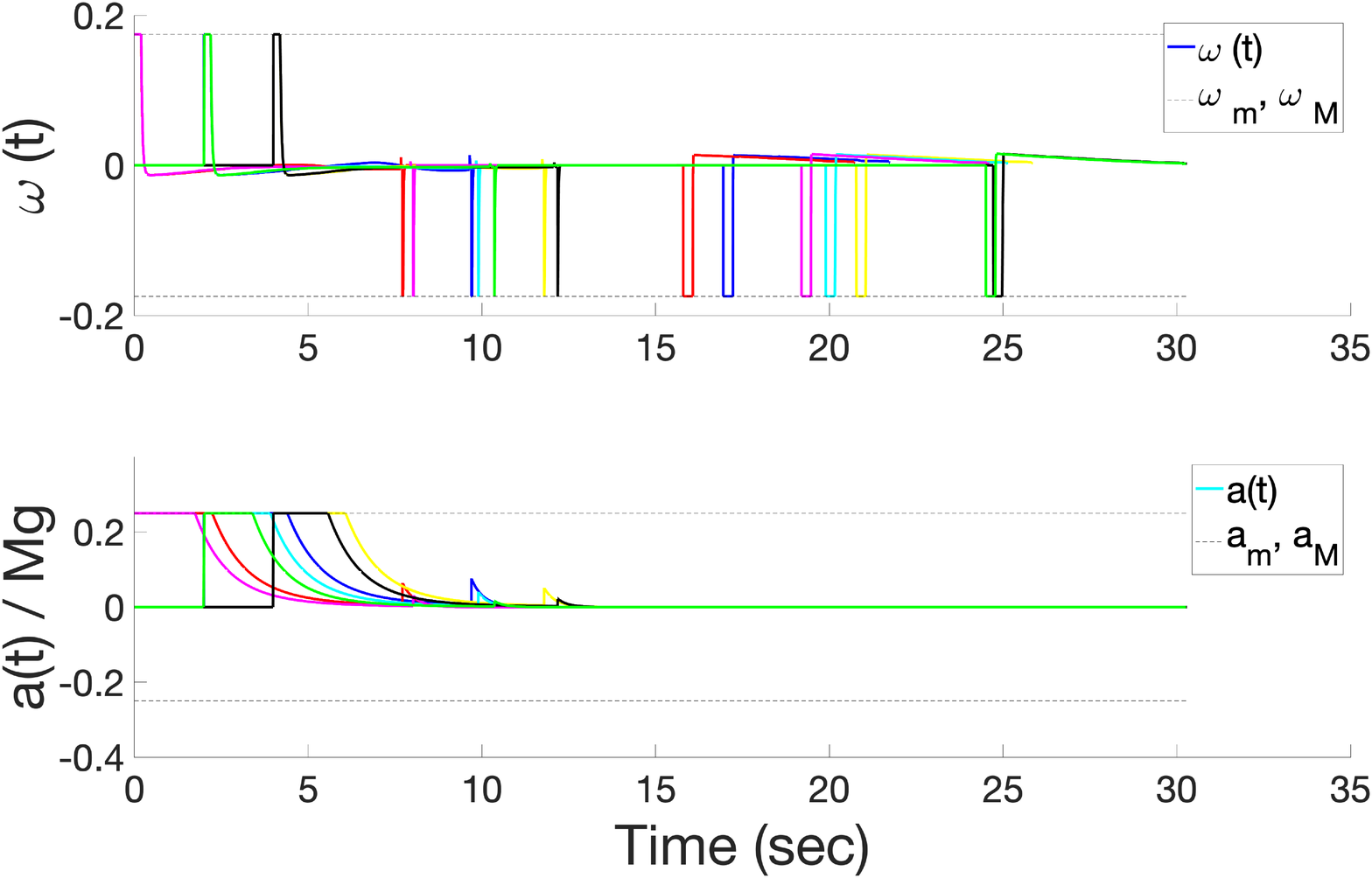}
        \includegraphics[width=1\columnwidth,clip]{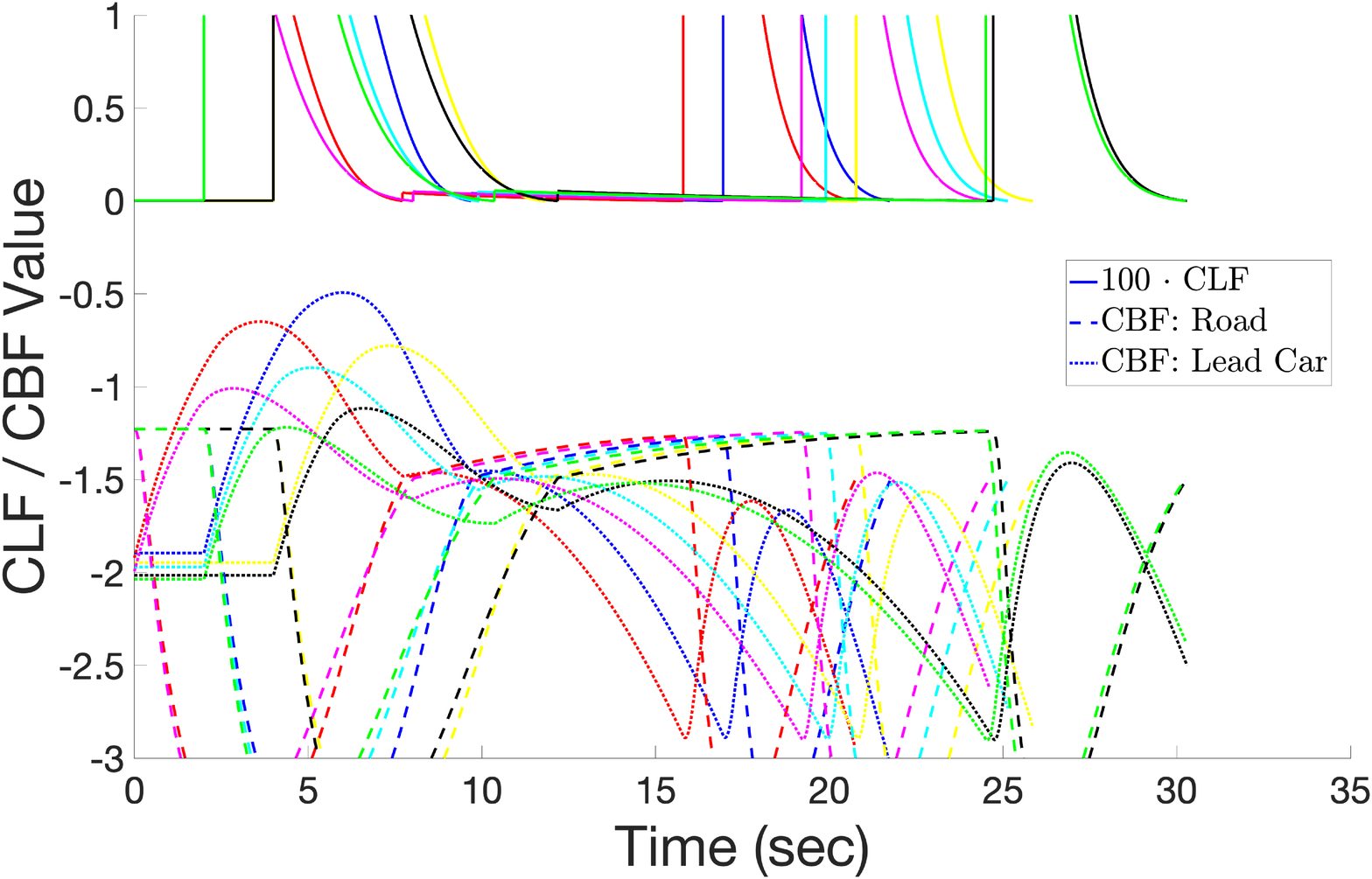}
    % \caption{Ego Vehicle trajectories, control inputs, and CLF / CBF evolution for 7 different initial conditions.}\label{fig:2}
    \caption{State trajectories, control trajectories, and CLF / CBF evolution respectively (top to bottom) of the Ego Vehicle during simulated scenarios using 7 different initial conditions. True CBF values have been negated for better visuals.}\label{fig:2}
\end{figure}
    
Figure \ref{fig:2} plots the paths traced by the Ego vehicle for various initial conditions $q_e(0)$. With the selected parameters for the QP \eqref{Robust FxT CLF-CBF QP}, it is clear from the figure that for all chosen initial conditions 1) the Ego car performed a successful maneuver and converged within the fixed-time windows; 2) the control inputs bounds are satisfied at all times; and 3) safety constraints are obeyed at all times. Additionally, in the case where the Oncoming Vehicle was scheduled to pass the Ego Vehicle at $t_p = 30$ sec, the Ego Vehicle appropriately made the decision to execute the overtake immediately.

Finally, 10 evenly spaced upper bounds on $\phi (q)$, from $0.1\|\phi\|_{\infty}$ to $1.0\|\phi\|_{\infty}$ are considered and the overtake maneuver is simulated. Figure \ref{fig:3} shows that for 100 trials of the perturbed simulation, 10 for each disturbance bound, the solutions of the individual sub-problems converged within the finite-time window. In Figure \ref{fig:4} we display the results for one such simulation. The two following observations are notable: 1) for $t_p = 30$ the safety estimator computed that Oncoming Vehicle was inside of the overtaking horizon, and as such decided not to initiate the overtake until after it passed; 2) consistent with \eqref{new FxTS T est nonvan}, as the disturbance bound grew so did the overtaking horizon - notably, when $t_p = 34$, the safety estimator computed that for $\|\phi\|_{\infty} = 0.4$, $1.6$, the Ego Vehicle could complete the overtake safely, whereas at larger disturbance bounds the decision to withhold the overtake was made until the Oncoming vehicle had passed safely by. Meanwhile, the controller satisfied the safety requirement for all trials.

\begin{figure}[!ht]
    \centering
        \includegraphics[width=1\columnwidth,clip]{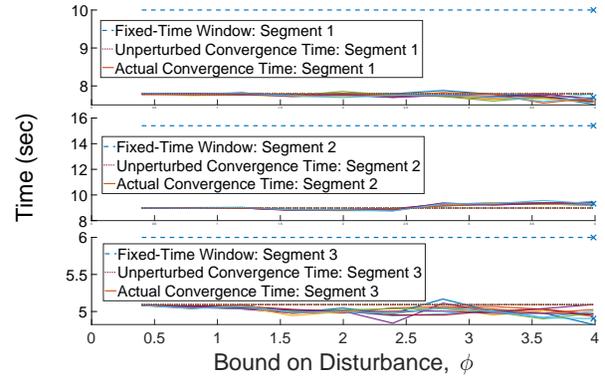}
    % \caption{Convergence of Nominal and Perturbed Solutions within Fixed-Time. }\label{fig:3}
    \caption{Time of convergence to goal set versus upper bound on disturbance term, $\phi$, for nominal and perturbed solutions, shown in conjunction with the fixed-time prescription per sub-problem.}\label{fig:3}
\end{figure}

\begin{figure}[!ht]
    \centering
        \includegraphics[width=1\columnwidth,clip]{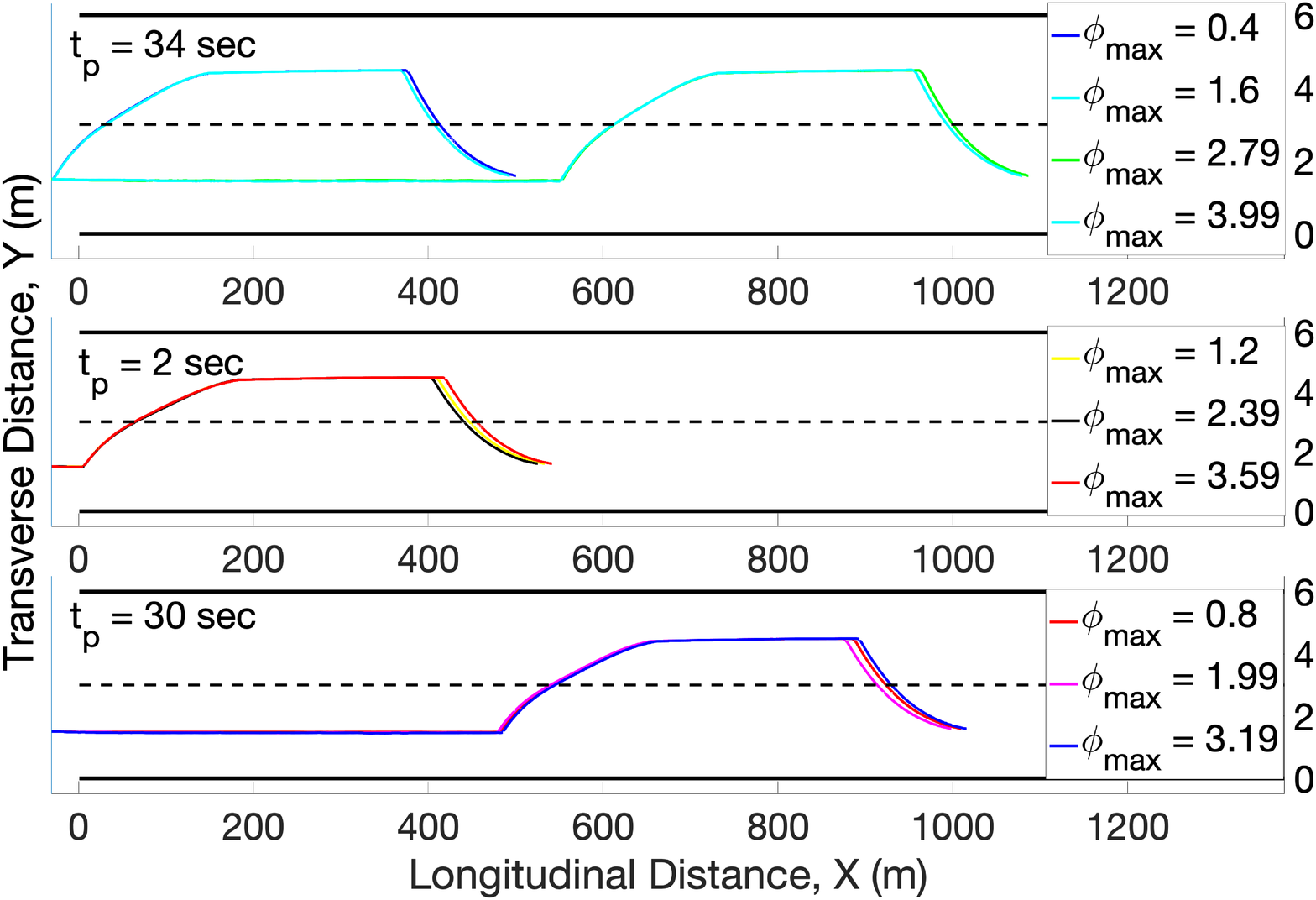}
        \includegraphics[width=1\columnwidth,clip]{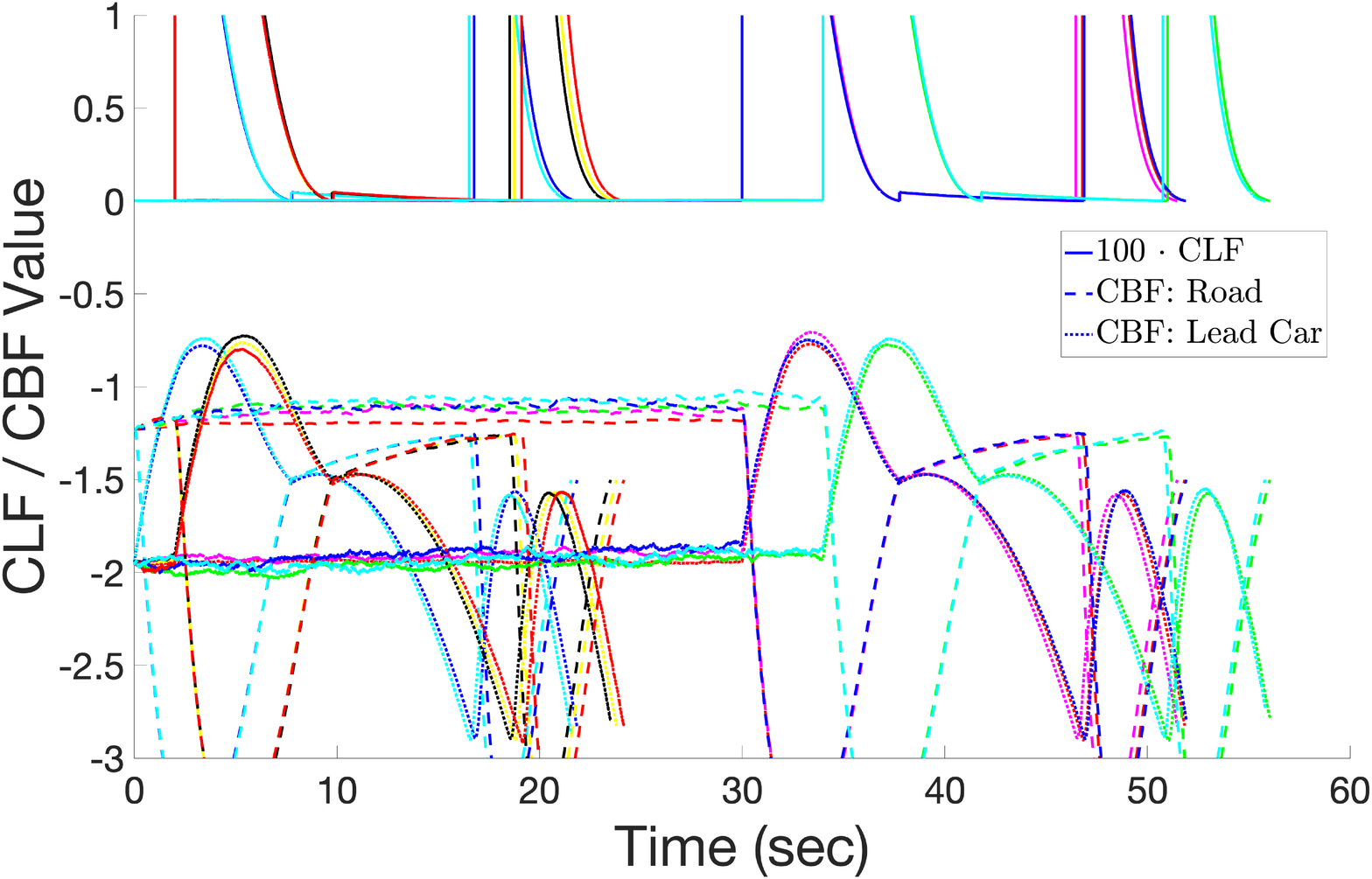}
    % \caption{Ego Vehicle trajectories and CLF / CBF evolution for case with 10 different disturbance bounds.}\label{fig:4}
    \caption{State trajectories and CLF / CBF evolution respectively (top to bottom) of the Ego Vehicle during simulated scenarios using 1 initial condition and 10 different disturbance bounds. True CBF values have been negated for better visuals.}\label{fig:4}
\end{figure}

\section{Conclusion}
\noindent In this study on robust control synthesis using CLF- and CBF-based techniques for safety-critical control problems, we introduced a new approach to driving a dynamical system subject to spatiotemporal and input constraints to a neighborhood of a goal set in fixed-time despite the presence of bounded, additive, non-vanishing disturbances. We provided theoretical guarantees of fixed-time convergence for such a system whose control is computed by a FxT-CLF-CBF QP provided that disturbances do not exceed a quantified bound. Next, we outlined a procedure for conditioning the QP and selecting parameters such that FxT convergence to a neighborhood of a goal set is guaranteed for any initial condition, and presented definitions for such a neighborhood. We then demonstrated the procedure on an overtake problem and highlighted the efficacy of the method with repeated simulated trials. In the future, we plan to explore reducing the conservativeness of this approach by considering an estimate of the non-vanishing disturbance term via online adaptation and/or learning based techniques.

\bibliographystyle{IEEEtran}
\bibliography{myreferences}

\begin{thebibliography}{10}
\providecommand{\url}[1]{#1}
\csname url@samestyle\endcsname
\providecommand{\newblock}{\relax}
\providecommand{\bibinfo}[2]{#2}
\providecommand{\BIBentrySTDinterwordspacing}{\spaceskip=0pt\relax}
\providecommand{\BIBentryALTinterwordstretchfactor}{4}
\providecommand{\BIBentryALTinterwordspacing}{\spaceskip=\fontdimen2\font plus
\BIBentryALTinterwordstretchfactor\fontdimen3\font minus
  \fontdimen4\font\relax}
\providecommand{\BIBforeignlanguage}[2]{{%
\expandafter\ifx\csname l@#1\endcsname\relax
\typeout{** WARNING: IEEEtran.bst: No hyphenation pattern has been}%
\typeout{** loaded for the language `#1'. Using the pattern for}%
\typeout{** the default language instead.}%
\else
\language=\csname l@#1\endcsname
\fi
#2}}
\providecommand{\BIBdecl}{\relax}
\BIBdecl

\bibitem{ames2017control}
A.~D. Ames, X.~Xu, J.~W. Grizzle, and P.~Tabuada, ``Control barrier function
  based quadratic programs for safety critical systems,'' \emph{IEEE Trans. on
  Automatic Control}, vol.~62, no.~8, pp. 3861--3876, 2017.

\bibitem{li2018formally}
A.~Li, L.~Wang, P.~Pierpaoli, and M.~Egerstedt, ``Formally correct composition
  of coordinated behaviors using control barrier certificates,'' in
  \emph{IEEE/RSJ International Conference on Intelligent Robots and
  Systems}.\hskip 1em plus 0.5em minus 0.4em\relax IEEE, 2018, pp. 3723--3729.

\bibitem{srinivasan2018control}
M.~Srinivasan, S.~Coogan, and M.~Egerstedt, ``Control of multi-agent systems
  with finite time control barrier certificates and temporal logic,'' in
  \emph{57th Conference on Decision and Control}, 2018, pp. 1991--1996.

\bibitem{ames2012control}
A.~D. Ames, K.~Galloway, and J.~W. Grizzle, ``Control {L}yapunov functions and
  hybrid zero dynamics,'' in \emph{51st Conference on Decision and
  Control}.\hskip 1em plus 0.5em minus 0.4em\relax IEEE, 2012, pp. 6837--6842.

\bibitem{romdlony2016stabilization}
M.~Z. Romdlony and B.~Jayawardhana, ``Stabilization with guaranteed safety
  using control {L}yapunov-barrier function,'' \emph{Automatica}, vol.~66, pp.
  39--47, 2016.

\bibitem{garg2019prescribed}
K.~{Garg}, E.~{Arabi}, and D.~{Panagou}, ``Prescribed-time {C}onvergence with
  {I}nput {C}onstraints: A {C}ontrol {L}yapunov {F}unction {B}ased
  {A}pproach,'' in \emph{2020 American Control Conference}, pp. 962--967.

\bibitem{ames2014rapidly}
A.~D. Ames, K.~Galloway, K.~Sreenath, and J.~W. Grizzle, ``Rapidly
  exponentially stabilizing control {L}yapunov functions and hybrid zero
  dynamics,'' \emph{IEEE Transactions on Automatic Control}, vol.~59, no.~4,
  pp. 876--891, 2014.

\bibitem{tee2009barrier}
K.~P. Tee, S.~S. Ge, and E.~H. Tay, ``Barrier {L}yapunov functions for the
  control of output-constrained nonlinear systems,'' \emph{Automatica},
  vol.~45, no.~4, pp. 918--927, 2009.

\bibitem{cortez2019control}
W.~S. Cortez, D.~Oetomo, C.~Manzie, and P.~Choong, ``Control barrier functions
  for mechanical systems: Theory and application to robotic grasping,''
  \emph{IEEE Transactions on Control Systems Technology}, 2019.

\bibitem{glotfelter2017nonsmooth}
P.~Glotfelter, J.~Cort{\'e}s, and M.~Egerstedt, ``Nonsmooth barrier functions
  with applications to multi-robot systems,'' \emph{IEEE Control Systems
  Letters}, vol.~1, no.~2, pp. 310--315, 2017.

\bibitem{lindemann2019control}
L.~Lindemann and D.~V. Dimarogonas, ``Control barrier functions for signal
  temporal logic tasks,'' \emph{IEEE Control Systems Letters}, vol.~3, no.~1,
  pp. 96--101, 2019.

\bibitem{polyakov2012nonlinear}
A.~Polyakov, ``Nonlinear feedback design for fixed-time stabilization of linear
  control systems,'' \emph{IEEE Transactions on Automatic Control}, vol.~57,
  no.~8, p. 2106, 2012.

\bibitem{garg2019prescribedTAC}
K.~Garg, E.~Arabi, and D.~Panagou, ``Fixed-time control under spatiotemporal
  and input constraints: A {Q}{P} based approach,'' \emph{arXiv preprint
  arXiv:1906.10091}, 2019.

\bibitem{nguyen2016orc}
Q.~Nguyen and K.~Sreenath, ``Optimal robust control for constrained nonlinear
  hybrid systems with application to bipedal locomotion,'' in \emph{American
  Control Conference}.\hskip 1em plus 0.5em minus 0.4em\relax IEEE, 2016, pp.
  4807--4813.

\bibitem{kolathaya2019isscbf}
S.~Kolathaya and A.~D. Ames, ``Input-to-state safety with control barrier
  functions,'' \emph{IEEE Control Systems Letters}, vol.~3, no.~1, pp.
  108--113, 2019.

\bibitem{khalil2002nonlinear}
H.~K. Khalil, \emph{Nonlinear systems}.\hskip 1em plus 0.5em minus 0.4em\relax
  Prentice hall Upper Saddle River, NJ, 2002, vol.~3.

\bibitem{xu2015robustness}
X.~Xu, P.~Tabuada, J.~W. Grizzle, and A.~D. Ames, ``Robustness of control
  barrier functions for safety critical control,'' \emph{IFAC-PapersOnLine},
  vol.~48, no.~27, pp. 54--61, 2015.

\bibitem{liberzon1999iss}
D.~{Liberzon}, ``Iss and integral-iss disturbance attenuation with bounded
  controls,'' in \emph{Proceedings of the 38th IEEE Conference on Decision and
  Control}, vol.~3, 1999, pp. 2501--2506.

\bibitem{Rajamani2012VDC}
R.~Rajamani, \emph{Vehicle Dynamics and Control}.\hskip 1em plus 0.5em minus
  0.4em\relax Springer US, 2012.

\bibitem{huang2018highwaymerging}
L.~Huang and D.~Panagou, ``Hierarchical design of highway merging controller
  using navigation vector fields under bounded sensing uncertainty,''
  \emph{Distributed Autonomous Robotic Systems}, vol.~9, pp. 341--356, 2018.

\end{thebibliography}

\appendices

\section{Proof of Lemma \ref{lemma:int dot V non van}}\label{app proof lemma int dot nonvan}
\begin{proof}
For $c_3<2\sqrt{c_1c_2}$, note that $-c_1V^{a_1}-c_2V^{a_2}+c_3\leq -2\sqrt{c_1c_2}V+c_3\leq -2\sqrt{c_1c_2}\bar V+c_3< 0$ for all $\bar V > \frac{c_3}{2\sqrt{c_1c_2}}$. So, choose $\bar V = 1$ so that the integrand is negative for all $V_0\geq \bar V = 1$. Using this, we obtain
\begin{align*}
     I = \int_{V_0 }^{1}\frac{dV}{-c_1V^{a_1}-c_2V^{a_2}+c_3}.
\end{align*}
Note that for $V\geq 1$, we have that $c_3\leq c_3V$. Using this, we obtain that 
\begin{align*}
     I \leq \int_{V_0 }^{1}\frac{dV}{-c_1V^{a_1}-c_2V^{a_2}+c_3V}.
\end{align*}
Using \cite[Lemma 1]{garg2019prescribedTAC}, we obtain that first expression in the above inequality evaluates to 
\begin{align*}
  \int_{V_0}^{1}\frac{dV}{-c_1V^{a_1}-c_2V^{a_2}+c_3V} \leq  \frac{\mu}{c_1k_1}\left(\frac{\pi}{2}-\tan^{-1}k_2\right),
\end{align*}
where  $k_1 = \sqrt{\frac{4c_1c_2-c_3^2}{4c_1^2}}$ and $k_2 = \frac{2c_1-c_3}{\sqrt{4c_1c_2-c_3^2}}$ which completes the proof of (i).

For the case when $c_3 \geq 2\sqrt{c_1c_2}$, we obtain that $\bar V  = (k\frac{c_3+\sqrt{c_3-4c_1c_2}}{2c_1})^\mu>(\frac{c_3+\sqrt{c_3-4c_1c_2}}{2c_1})^\mu)>1$ for any $k>1$. Thus, for $V\geq \bar V > 1$, we have that $-c_1V^{a_1}-c_2 V^{a_2}+c_3\leq-c_1V^{a_1}-c_2 V^{a_2}+c_3V$, using which, we obtain that
\begin{align*}
    I \leq \int_{V_0}^{\bar V}\frac{dV}{-c_1V^{a_1}-c_2V^{a_2}+c_3V}.
\end{align*}
We obtain that \small{
\begin{align*}
  & \int_{V_0}^{\bar V}\frac{dV}{-c_1V^{a_1}-c_2V^{a_2}+c_3V}\\
  & \leq \frac{-\mu}{c_1(b-a)}\left(\log\left(\frac{|\bar V^\frac{1}{\mu}-a|}{|V_0^\frac{1}{\mu}-a|}\right)-\log\left(\frac{|\bar V^\frac{1}{\mu}-b|}{|V_0^\frac{1}{\mu}-b|}\right)\right)\\
  & = \frac{\mu}{c_1(b-a)}\left(\log\left(\frac{|\bar V^\frac{1}{\mu}-b|}{|\bar V^\frac{1}{\mu}-a|}\right)+\log\left(\frac{|V_0^\frac{1}{\mu}-a|}{|V_0^\frac{1}{\mu}-b|}\right)\right)\\
  & \leq \frac{\mu}{c_1(b-a)}\log\frac{|\bar V^\frac{1}{\mu}-b|}{|\bar V^\frac{1}{\mu}-a|} = \frac{\mu}{c_1(b-a)}\log\left(\frac{kb-b}{kb-a}\right),
\end{align*}}\normalsize
since the term corresponding to $V_0$ is less or equal to zero since $a\leq b$, and $\bar V = (kb)^\mu$.

% For the case when $c_3 = 2\sqrt{c_1c_2}$, we have that $a = b$, and so, using the analysis in the of proof \cite[Lemma 1]{garg2019prescribedTAC}, we obtain that \small{
% \begin{align*}
%     I \leq \int_{V_0 }^{\bar V}\frac{dV}{-c_1V^{a_1}-c_2V^{a_2}+c_3V} & = \frac{\mu}{c_1}\Big(\frac{1}{a+\bar V^\frac{1}{\mu}}-\frac{1}{a+V_0^\frac{1}{\mu}}\Big)\\
%     \leq \frac{\mu}{c_1}\frac{1}{a+\bar V^\frac{1}{\mu}} & = \frac{\mu}{\sqrt{c_2c_1}(k-1)},
% \end{align*}}\normalsize
% where the last inequality follows using the fact that $a = \frac{-c_3}{2c_1} = -\sqrt{\frac{c_2}{c_1}}$. This completes the proof.
\end{proof}

\end{document}